\documentclass[11pt]{article}

\usepackage[a4paper,margin=1in]{geometry}
\usepackage{amsmath,amssymb,amsthm}
\usepackage{mathtools}
\usepackage{booktabs}
\usepackage{comment}
\usepackage{xcolor}
\usepackage{todonotes}
\usepackage[hidelinks]{hyperref}

\hypersetup{
  pdftitle={Proof of conjecture of Ran and Teng on spectra of \(4 \times 4\) tridiagonal stochastic matrices nonnegativity of the second largest eigenvalue of \(4 \times 4\) tridiagonal stochastic matrices},
  pdfauthor={Brando Vagenende, Brecht Verbeken, Marie-Anne Guerry, and Andres Algaba}
}

\newtheorem{theorem}{Theorem}[section]
\newtheorem{lemma}[theorem]{Lemma}
\newtheorem{conjecture}[theorem]{Conjecture}

\newenvironment{keywords}
{\par\medskip\noindent\textbf{Keywords.}\ }
{\par\medskip}

\newenvironment{AMS}
{\par\medskip\noindent\textbf{AMS subject classifications.}\ }
{\par\medskip}

\title{Nonnegativity of the second largest eigenvalue of \(4 \times 4\) tridiagonal stochastic matrices}
\author{Brando Vagenende, Brecht Verbeken, Marie-Anne Guerry \footnote{Department Business Technology and Operations, Data Analytics Laboratory, Vrije Universiteit Brussel (VUB), Pleinlaan 2, Brussels, 1050, Belgium (brando.vagenende@vub.be,
brecht.verbeken@vub.be, marie-anne.guerry@vub.be).}}
\date{}

\begin{document}

\maketitle

\begin{abstract}
The spectral study of nonnegative and more specifically stochastic matrices is an important topic in matrix theory. In this paper, we prove a conjecture, formulated by Ran and Teng \cite{ran2024nonnegative}, which states that the second largest eigenvalue of an irreducible $4\times4$ tridiagonal stochastic matrix is nonnegative. We establish this conjecture and extend the result to arbitrary $4\times4$ tridiagonal stochastic matrices, including both irreducible and reducible cases.
\end{abstract}



\section{Introduction}
Recent research concerns the investigation of the eigenvalues of stochastic matrices and their subsets. The tridiagonal stochastic matrices form one of these stochastic subsets.

Ran and Teng~\cite{ran2024nonnegative} studied spectral regions for row-stochastic matrices with prescribed zero patterns. The $3\times3$ tridiagonal stochastic matrices, for $\alpha,\beta,\gamma, \delta \in[0,1]$ with $\beta + \gamma \leq 1$, 
\[
\begin{pmatrix}
\alpha & 1-\alpha & 0 \\
\beta & \gamma & 1-\beta-\gamma \\
0 & 1-\delta & \delta\\
\end{pmatrix}
\]
have the same spectral region as the set of stochastic matrices with two zeros
\[
\begin{pmatrix}
\alpha & 0 & 1-\alpha\\
0 & \delta & 1-\delta\\
\beta & 1-\beta-\gamma & \gamma\\
\end{pmatrix}.
\]
of which the spectral region is fully determined in (\cite{ran2024nonnegative}, Proposition 17).

Further, for \(4 \times 4\) tridiagonal stochastic matrices
\begin{equation}\label{eq:A}
A=
\begin{pmatrix}
\alpha & 1-\alpha & 0 & 0\\
\beta & \gamma & 1-\beta-\gamma & 0\\
0 & \delta & \phi & 1-\delta-\phi\\
0 & 0 & 1-\kappa & \kappa
\end{pmatrix}
\end{equation}
where \(0 \le \alpha,\beta,\gamma,\delta,\phi,\kappa \le 1\), \(\beta+\gamma \le 1\), and \(\delta+\phi \le 1\), Ran and Teng stated in \cite{ran2024nonnegative} the following conjecture, knowing that irreducible tridiagonal matrices have only real eigenvalues.

\begin{conjecture}\label{conjecture}
Let \(A\) be an irreducible \(4\times 4\) tridiagonal stochastic matrix of form (\ref{eq:A}), and let its eigenvalues be ordered as
\(
\lambda_1(A) = 1 \ge \lambda_2(A) \ge \lambda_3(A) \ge \lambda_4(A).
\)
Then the second largest eigenvalue satisfies \(\lambda_2(A) \ge 0\).
\end{conjecture}

In Section \ref{IrreducibleCase}, we prove this conjecture. After this, in Section \ref{ReducibleCase}, we extend this result to all \(4 \times 4\) tridiagonal stochastic matrices of the form (\ref{eq:A}), including irreducible as well as reducible tridiagonal stochastic matrices.

\section{Irreducible case}\label{IrreducibleCase}
In this section we prove Conjecture \ref{conjecture}, which is stated for $4\times4$ tridiagonal stochastic matrices that are irreducible. We first characterize irreducibility:

\begin{lemma}\label{lemma:irreducible}
The tridiagonal stochastic matrix \(A\) in (\ref{eq:A}) is irreducible if and only if its super- and subdiagonal matrix entries are stricly positive, i.e.
\(
1 - \alpha > 0,
1 - \beta - \gamma > 0,
1 - \delta - \phi > 0,
\beta > 0,
\delta > 0 \text{ and }
1 - \kappa > 0.
\)
\end{lemma}
\begin{proof}
The statement follows from the fact that the graph associated with the tridiagonal matrix $A$ is strongly connected if and only if its super- and subdiagonal entries are stricly positive. Alternatively, one can easily verify that $I + A + A^2 + A^3 > 0$ iff 1 - $\alpha > 0, 1 - \beta - \gamma > 0, 1 - \delta - \phi > 0, \beta > 0, \delta > 0 \text{ and } 1 - \kappa > 0$.
\end{proof}

The assumption that \(A\) of the form (\ref{eq:A}) is irreducible, allows to construct the following similar symmetric matrix \(S\).
\begin{lemma}
    An irreducible tridiagonal stochastic matrix \(A\) of the form (\ref{eq:A}) is similar to the symmetric matrix
\[
S=
\begin{pmatrix}
\alpha & r_1 & 0 & 0\\
r_1 & \gamma & r_2 & 0\\
0 & r_2 & \phi & r_3\\
0 & 0 & r_3 & \kappa
\end{pmatrix},
\]
where \(r_1=\sqrt{(1-\alpha) \beta},
r_2=\sqrt{(1-\beta-\gamma) \delta} \text{ and }
r_3=\sqrt{(1-\delta-\phi)(1-\kappa)}\).
\end{lemma}

Consequently, \(A\) and \(S\) have the same (real) spectrum.

For \(k=1,2,3,4\), let \(S_k\) denote the leading \(k\times k\) principal submatrix of \(S\), that is, the submatrix obtained by restricting \(S\) to its first \(k\) rows and first \(k\) columns. Define \(\Delta_k := \det(S_k)\) with \(\Delta_0 := 1\).
\begin{lemma}\label{lem:inertia}
Assume $\Delta_k\neq 0$ for $k=1,2,3,4$. Then the number of negative eigenvalues of $S$
equals the number of sign changes in the sequence \((\Delta_0,\Delta_1,\Delta_2,\Delta_3,\Delta_4).\)
\end{lemma}
\begin{proof}
Since $\Delta_k\neq 0$ for $k=1,2,3,4$, the symmetric matrix $S$ can be transformed through Gauss elimination into the diagonal matrix $D=\text{diag}(\frac{\Delta_0}{\Delta_1},\frac{\Delta_1}{\Delta_2},\frac{\Delta_2}{\Delta_3},\frac{\Delta_3}{\Delta_4})$ (\cite{gantmakher2000theory}, Chapter X). Moreover, The Law of Inertia guarantees that the matrices $S$ and $D$ have the same number of negative eigenvalues. For $D$, an eigenvalue $\frac{\Delta_i}{\Delta_{i+1}}$ is negative iff $\Delta_i$ and $\Delta_{i+1}$ have different signs. The result follows.
\end{proof}
A straightforward computation yields
\begin{align*}
\Delta_1 &= \alpha \ge 0,\\
\Delta_2 &= \alpha\gamma - r_1^2,\\
\Delta_3 &= \phi\Delta_2 - r_2^2\Delta_1,\\
\Delta_4 &= \kappa\Delta_3 - r_3^2\Delta_2,
\end{align*}
which gives the following lemma.

\begin{lemma}\label{lem:imp}
The following statements hold:
\begin{enumerate}
\item If \(\Delta_2 \le 0\), then \(\Delta_3 \le 0\).
\item If \(\Delta_3 \le 0\) and \(\Delta_2 \ge 0\), then \(\Delta_4 \le 0\).
\end{enumerate}
\end{lemma}

\begin{proof}
(1) Since \(\phi \ge 0\) and \(\Delta_1 \ge 0\), we have
\[
\Delta_3 = \phi\Delta_2 - r_2^2\Delta_1 \le \phi\Delta_2 \le 0.
\]

\noindent
(2) Since \(\kappa \ge 0\), it follows that
\[
\Delta_4 = \kappa\Delta_3 - r_3^2\Delta_2 \le 0.
\]
\end{proof}

We now prove Conjecture \ref{conjecture}, although the result is obtained under an extra assumption.

\begin{lemma}\label{prop:atmost2neg_generic}
Assume \(\Delta_k \neq 0\) for \(k=1,2,3,4\).
Then \(S\) (and hence \(A\)) has at most two negative eigenvalues.
\end{lemma}

\begin{proof}
We consider the possible sign patterns of the sequence
\[
(\Delta_0,\Delta_1,\Delta_2,\Delta_3,\Delta_4),
\qquad \Delta_0=1>0,\ \Delta_1=\alpha\ge 0.
\]

If \(\Delta_2<0\), then Lemma~\ref{lem:imp} implies that \(\Delta_3<0\).
In this case, sign changes can only occur at \(\Delta_1\to\Delta_2\) and possibly at \(\Delta_3\to\Delta_4\), yielding at most two sign changes.

If \(\Delta_2>0\), then either \(\Delta_3<0\) or \(\Delta_3>0\).
If \(\Delta_3<0\), Lemma~\ref{lem:imp} gives \(\Delta_4<0\), so there is exactly one sign change.
If \(\Delta_3>0\), any sign change can only occur at \(\Delta_3\to\Delta_4\), and hence there is at most one sign change.

In all cases, the sequence \((\Delta_0,\Delta_1,\Delta_2,\Delta_3,\Delta_4)\) has at most two sign changes.
By Lemma~\ref{lem:inertia}, it follows that \(S\) has at most two negative eigenvalues.
\end{proof}

The following lemma allows us to remove the assumption \(\Delta_k \neq 0\) by a limiting argument.

\begin{lemma}\label{benadering}
Let $A$ be an irreducible tridiagonal stochastic $4 \times 4$ matrix.
Then there exists a sequence $A^{(n)} \to A$ such that each $A^{(n)}$
is irreducible, tridiagonal and stochastic, and such that for the corresponding
symmetrised matrices $S^{(n)}$ one has
\[
\Delta_k(S^{(n)}) \neq 0, \qquad k = 1,2,3,4.
\]
\end{lemma}

\begin{proof}
Write
\[
A = A(\alpha,\beta,\gamma,\delta,\phi,\kappa)
=
\begin{pmatrix}
\alpha & 1-\alpha & 0 & 0 \\
\beta & \gamma & 1-\beta-\gamma & 0 \\
0 & \delta & \phi & 1-\delta-\phi \\
0 & 0 & 1-\kappa & \kappa
\end{pmatrix}.
\]
Since $A$ is stochastic and irreducible, its parameters satisfy
\[
0 \leq \alpha < 1, \qquad
\beta > 0, \qquad
0 \leq \gamma < 1-\beta,
\]
\[
\delta > 0, \qquad
0 \leq \phi < 1-\delta, \qquad
0 \leq \kappa < 1.
\]

For $n \geq 1$, set
\[
\varepsilon_n = \frac{1}{n+1}.
\]
We shall construct
\[
v^{(n)}
=
(\alpha^{(n)},\beta,\gamma^{(n)},\delta,\phi^{(n)},\kappa^{(n)})
\]
with
\[
v^{(n)} \to (\alpha,\beta,\gamma,\delta,\phi,\kappa),
\]
while keeping all parameters inside the admissible irreducible stochastic
region.

First choose
\[
\alpha^{(n)}
=
\begin{cases}
\alpha, & \alpha > 0,\\
\varepsilon_n, & \alpha = 0.
\end{cases}
\]
Then
\[
0 < \alpha^{(n)} < 1,
\]
so
\[
\Delta_1(S^{(n)}) = \alpha^{(n)} \neq 0.
\]

Next, with $\alpha^{(n)}$ fixed, consider $\Delta_2$ as a function of the
second diagonal parameter:
\[
F_{2,n}(x)
=
\alpha^{(n)}x - (1-\alpha^{(n)})\beta.
\]
This is a nonconstant affine function because $\alpha^{(n)} > 0$. Hence it has
exactly one zero. The interval
\[
I_{\gamma,n}
=
(0,1-\beta) \cap (\gamma-\varepsilon_n,\gamma+\varepsilon_n)
\]
contains infinitely many points. Therefore we may choose
\[
\gamma^{(n)} \in I_{\gamma,n}
\]
such that
\[
F_{2,n}(\gamma^{(n)}) \neq 0.
\]
Then
\[
\Delta_2(S^{(n)}) \neq 0.
\]
Moreover,
\[
|\gamma^{(n)}-\gamma| < \varepsilon_n,
\]
and the inequalities
\[
\gamma^{(n)} > 0,
\qquad
\beta+\gamma^{(n)} < 1
\]
ensure that stochasticity and irreducibility are preserved in the second row.

Now fix $\alpha^{(n)}$ and $\gamma^{(n)}$. We have already arranged
\[
\Delta_1(S^{(n)}) \neq 0,
\qquad
\Delta_2(S^{(n)}) \neq 0.
\]
Since
\[
r_{2,n}^2 = (1-\beta-\gamma^{(n)})\delta > 0,
\]
the third leading principal minor, as a function of the third diagonal
parameter, is
\[
F_{3,n}(x)
=
\Delta_2(S^{(n)})x
-
r_{2,n}^2\Delta_1(S^{(n)}).
\]
This is again a nonconstant affine function because
\(
\Delta_2(S^{(n)}) \neq 0.
\)
Hence it has exactly one zero. The interval
\[
I_{\phi,n}
=
(0,1-\delta) \cap (\phi-\varepsilon_n,\phi+\varepsilon_n)
\]
contains infinitely many points. Choose
\[
\phi^{(n)} \in I_{\phi,n}
\]
such that
\[
F_{3,n}(\phi^{(n)}) \neq 0.
\]
Then
\[
\Delta_3(S^{(n)}) \neq 0.
\]
Also
\[
|\phi^{(n)}-\phi| < \varepsilon_n,
\qquad
\phi^{(n)} > 0,
\qquad
\delta+\phi^{(n)} < 1.
\]

Finally fix $\alpha^{(n)}$, $\gamma^{(n)}$, $\phi^{(n)}$. Put
\[
r_{3,n}^2(x)
=
(1-\delta-\phi^{(n)})(1-x).
\]
For the fourth leading principal minor, as a function of the final diagonal
parameter $x$, we get
\[
F_{4,n}(x)
=
x\Delta_3(S^{(n)})
-
(1-\delta-\phi^{(n)})(1-x)\Delta_2(S^{(n)}).
\]
This is an affine function of $x$. It is not identically zero, because
\[
F_{4,n}(0)
=
-(1-\delta-\phi^{(n)})\Delta_2(S^{(n)}) \neq 0.
\]
Hence $F_{4,n}$ has at most one zero. The interval
\[
I_{\kappa,n}
=
(0,1) \cap (\kappa-\varepsilon_n,\kappa+\varepsilon_n)
\]
contains infinitely many points. Therefore choose
\[
\kappa^{(n)} \in I_{\kappa,n}
\]
such that
\[
F_{4,n}(\kappa^{(n)}) \neq 0.
\]
Then
\[
\Delta_4(S^{(n)}) \neq 0.
\]

Define
\[
A^{(n)}
=
A(\alpha^{(n)},\beta,\gamma^{(n)},\delta,\phi^{(n)},\kappa^{(n)}).
\]
By construction,
\[
0 < \alpha^{(n)} < 1, \qquad
\beta > 0, \qquad
0 < \gamma^{(n)} < 1-\beta,
\]
\[
\delta > 0, \qquad
0 < \phi^{(n)} < 1-\delta, \qquad
0 < \kappa^{(n)} < 1.
\]
Thus each $A^{(n)}$ is stochastic, tridiagonal and irreducible. Furthermore,
\[
|\alpha^{(n)}-\alpha| \leq \varepsilon_n,
\qquad
|\gamma^{(n)}-\gamma| < \varepsilon_n,
\]
\[
|\phi^{(n)}-\phi| < \varepsilon_n,
\qquad
|\kappa^{(n)}-\kappa| < \varepsilon_n,
\]
while $\beta$ and $\delta$ are unchanged. Hence
\[
(\alpha^{(n)},\beta,\gamma^{(n)},\delta,\phi^{(n)},\kappa^{(n)})
\to
(\alpha,\beta,\gamma,\delta,\phi,\kappa).
\]
Since the entries of $A(v)$ depend continuously on the parameter vector $v$, it
follows that
\[
A^{(n)} \to A
\]
entrywise. By the choices above,
\[
\Delta_k(S^{(n)}) \neq 0,
\qquad k=1,2,3,4.
\]
This proves the lemma.
\end{proof}

\begin{theorem}\label{prop:atmost2neg}
Let $A$ be a $4\times 4$ tridiagonal stochastic matrix. Assume that \(A\) is irreducible. Then $A$ has at most two negative eigenvalues.
\end{theorem}

\begin{proof}
Let $S^{(n)}$ be as in Lemma \ref{benadering}.
By Lemma \ref{prop:atmost2neg_generic}, each $S^{(n)}$ has at most two
negative eigenvalues.

Assume for contradiction that $S$ has at least three negative eigenvalues.
Since $S$ is symmetric, we may order its eigenvalues
$\lambda_1(S)\ge\cdots\ge \lambda_4(S)$.
The assumption implies $\lambda_3(S)<0$.
Hence there exists $\eta>0$ such that $\lambda_3(S)\le -\eta$.

Eigenvalues of symmetric matrices depend continuously on the entries.
Therefore, for all sufficiently large $n$ we have
$\lambda_3(S^{(n)})<0$, so $S^{(n)}$ would have at least three
negative eigenvalues, which results in a contradiction.
Thus $S$ has at most two negative eigenvalues. Because $A$ and $S$ share the same eigenvalues, the proposition follows.
\end{proof} 
Hence, Conjecture \ref{conjecture} holds and the nonnegativity of the second largest eigenvalue for irreducible $4\times 4$ tridiagonal stochastic matrices is proved.

\section{General case}\label{ReducibleCase}
In this section we remove the irreducibility assumption and prove Conjecture~\ref{conjecture} for arbitrary $4\times 4$ tridiagonal stochastic matrices. The reducible case is obtained by an approximation argument: we perturb a given reducible matrix within the class of tridiagonal stochastic matrices to obtain a sequence of irreducible matrices, apply the result of the previous section to this sequence, and then pass to the limit using continuity of the eigenvalues. This yields the desired nonnegativity of the second largest eigenvalue in full generality.

\begin{lemma}\label{lem:irreducible_approx}
For any parameters \((\alpha, \beta, \gamma, \delta, \phi, \kappa)\) satisfying the conditions for (\ref{eq:A}) and any $\varepsilon\in(0,1)$,
define
\begin{align*}
\alpha_\varepsilon &= (1-\varepsilon)\alpha + \varepsilon/2,\\
\beta_\varepsilon  &= (1-\varepsilon)\beta  + \varepsilon/2,\qquad
\gamma_\varepsilon = (1-\varepsilon)\gamma,\\
\delta_\varepsilon &= (1-\varepsilon)\delta + \varepsilon/2,\qquad
\phi_\varepsilon   = (1-\varepsilon)\phi,\\
\kappa_\varepsilon &= (1-\varepsilon)\kappa + \varepsilon/2.
\end{align*}
Let $A_\varepsilon$ be built from these parameters as in (\ref{eq:A}).
Then:
\begin{enumerate}
\item $A_\varepsilon$ has nonnegative entries and row sums $1$, i.e. $A_\varepsilon$ is stochastic.
\item $A_\varepsilon$ is irreducible.
\item $A_\varepsilon\to A(\alpha, \beta, \gamma, \delta, \phi, \kappa)$ entrywise as $\varepsilon\downarrow 0$.
\end{enumerate}
\end{lemma}

\begin{proof}
Row-stochasticity follows by construction.
We check the constraints.

First, $0\le \alpha_\varepsilon,\beta_\varepsilon,\gamma_\varepsilon,
\delta_\varepsilon,\phi_\varepsilon,\kappa_\varepsilon\le 1$.
Next,
\[
\beta_\varepsilon+\gamma_\varepsilon
=(1-\varepsilon)(\beta+\gamma)+\varepsilon/2
\le (1-\varepsilon)+\varepsilon/2
=1-\varepsilon/2<1,
\]
so $1-\beta_\varepsilon-\gamma_\varepsilon\ge \varepsilon/2>0$.
Similarly,
\[
\delta_\varepsilon+\phi_\varepsilon
=(1-\varepsilon)(\delta+\phi)+\varepsilon/2
\le 1-\varepsilon/2<1,
\]
so $1-\delta_\epsilon-\phi_\epsilon\ge \varepsilon/2>0$.

We also have
$\beta_\varepsilon\ge \varepsilon/2>0$,
$\delta_\varepsilon\ge \varepsilon/2>0$,
$1-\kappa_\varepsilon
=1-(1-\varepsilon)\kappa-\varepsilon/2
\ge 1-(1-\varepsilon)-\varepsilon/2
= \varepsilon/2 >0$.
Finally,
\[
1-\alpha_\varepsilon
=1-(1-\varepsilon)\alpha-\varepsilon/2
\ge 1 - (1-\varepsilon) - \varepsilon/2 = \varepsilon/2 > 0.
\]
Thus $A_\varepsilon$ is irreducible according to Lemma \ref{lemma:irreducible}.
The convergence $A_\varepsilon\to A$ is immediate from the definitions.
\end{proof}

\begin{theorem}\label{thm:main}
For every matrix $A$ as in (\ref{eq:A}),
\[
\lambda_2(A)\ge 0.
\]
\end{theorem}

\begin{proof}
Either $A$ is irreducible, and then the result follows from Theorem \ref{prop:atmost2neg}. Or $A$ is reducible, and then for each $\varepsilon\in(0,1)$, the matrix $A_\varepsilon$ from
Lemma~\ref{lem:irreducible_approx} is irreducible,
so
\[
\lambda_2(A_\varepsilon)\ge 0.
\]

Assume for contradiction that $\lambda_2(A)<0$.
Then $\lambda_2(A)$ is a real negative eigenvalue. Since we order eigenvalues
by value,
\[
\lambda_2(A)\ge \lambda_3(A)\ge \lambda_4(A)
\]
implies that $A$ has three (real) negative eigenvalues.

Eigenvalues are roots of the characteristic polynomial and the coefficients of
this polynomial depend continuously on the matrix entries.
Hence there exists $\varepsilon_0>0$ such that for all $\varepsilon<\varepsilon_0$,
the matrix $A_\varepsilon$ has three eigenvalues close to those
three negative eigenvalues of $A$, and has in particular three negative
eigenvalues. This would force $\lambda_2(A_\varepsilon)<0$, contradicting
$\lambda_2(A_\varepsilon)\ge 0$.

Therefore, we can conclude $\lambda_2(A)\ge 0$.
\end{proof}

\section{Conclusions and further research}\label{Conclusions and further research}

In this paper, we prove the conjecture of Ran and Teng stating that the second largest eigenvalue of an irreducible \(4\times4\) tridiagonal stochastic matrix is nonnegative. Using an approximation argument, we subsequently extended the result to all \(4\times4\) tridiagonal stochastic matrices, including reducible matrices. Hence, for every matrix of the form (\ref{eq:A}), the property
\[
\lambda_2(A)\ge 0
\]
holds.

A natural direction for further research is the investigation of higher-dimensional tridiagonal stochastic matrices. In particular, it would be interesting to determine whether analogous spectral restrictions hold for \(n\times n\) tridiagonal stochastic matrices with \(n\ge 5\), and whether similar techniques based on symmetrization and inertia can be extended to these cases.

\bibliographystyle{plain}
\bibliography{references}
\end{document}